\documentclass[reqno]{amsart}
\usepackage{amsmath, bm, amsthm, amssymb}
\usepackage{setspace}
\usepackage{graphicx}
\usepackage{graphics, epsfig, subfigure, epstopdf}
\usepackage{curves}

\theoremstyle{plain}
\newtheorem{thm}{Theorem}[section]
\newtheorem{lemma}{Lemma}[section]

\theoremstyle{remark}
\newtheorem{rema}{Remark}[section]
\newtheorem{example}{Example}[section]

\numberwithin{equation}{section}

\doublespace

\begin{document}
\title[Inverse Scattering Problem with Multi-frequency Data]{Inverse Scattering Approach on Tomography Problem Using Multi-frequency Data}

\author[Ying Li]{Ying Li\\
Department of Mathematics\\
Saint Francis University\\
Loretto, PA 15940\\
814-4723854\\
YLi@francis.edu}


\begin{abstract}

An inverse scattering problem is formulated for reconstructing optical properties of biological tissues. A recursive linearization algorithm is used
to solve the inverse scattering problem. We employed the idea of
finite element boundary integral method and added suitable boundary
conditions on the surface of the domain. The initial guess is obtained by Born approximation based on the fact of weak scattering. The reconstruction 
is then improved each time by an increment on wave number. Finite element method is used for the
interior domain containing inhomogeneity. Nystr$\ddot{o}$m method is
used for setting up the boundary conditions and jump conditions. Two numerical
examples are presented.

\end{abstract}

\maketitle

\section{Introduction}
Photo-acoustic tomography has been shown great interest over the past decades and used to reconstruct optical property of biological tissue, such as the breast and the brain. It is a phenomenon in which the optical properties of the underlying medium is modified by absorbed radiation which in turn generates measurable acoustic waves. The acoustic signal is then collected to recover the property of the medium. Readers are referred to~\cite{B-W} for a description of the photo-acoustic effect and~\cite{H-M-P-D, K-L-F-A, W-P-K-X-S-W, X-W, Z-L-B} for the development of the hybrid imaging modality, which combines the optical methods with the spatial resolution of ultrasound imaging.

The difficulty of acousto-optic imaging lies in various aspects of the reconstruction of absorbed radiation from the acoustic signal, such as limited data, spatially varying acoustic sound speed and the effects of acoustic wave attenuation. In~\cite{B-R-U-Z}, the authors assume the absorbed radiation known, then reconstructed the conductivity coefficient from the known absorbed radiation.

In this paper, we formulate the problem as an inverse scattering
problem, which is to determine the conductivity property of the tissue
from the measurements of electromagnetic field on the boundary
of the medium, given the incident field.

Some related results can be found in~\cite{B-K-book} and~\cite{B-K-Journal}, where the authors proposed a globally convergent numerical method
for an inverse problem of recovering the coefficient from non overdetermined time dependent
data with the single source location. The book~\cite{B-K-book} summarized results of its authors published in various journals
 in 2008-2011. Their algorithm was tested on both
computationally simulated and experimental data.

Our approach follows the general
idea of~\cite{ChenYu} and employs the recursive linearization
algorithm from~\cite{peijun} and~\cite{B-L}. In their work the authors used fixed-frequency data, which contains multiple spatial frequency evanescent plane waves. In this paper multi-frequency data is used to recover the dielectric property of 
a dispersive medium, which depends on the wavenumber and has wider applications. A relevant convergence result can be found in ~\cite{B-T}, in which the authors proved the convergence of the algorithm along with an error estimate under some reasonable assumptions.  

In two dimensional cases, the electromagnetic intensity satisfies
the Helmholtz equation:

\begin{align}
\Delta u+k_0^2(1+q(k_0,x)) u=0,\label{helmeq}
\end{align}
where $u$ is the total field; $k_0$ is the wavenumber in vacuum; $q(k_0,x)$ is the scatterer which has a
compact support and $\epsilon(k_0, x) = 1+ q(k_0, x)$ is the dielectric permittivity
in dispersive medium. We assumed that $q(k_0,x)={\rm i} \sigma(x)/k_0$, where
$\sigma(x)$ is the conductivity of the medium.
In the following, we
assume that the material is nonmagnetic, i.e., $\mu_0=1$.

The scatterer is illuminated by a one-parameter family of plane
waves

\begin{align}
u^{\rm i}=\exp({\rm i} k_0 \bm x \cdot {\bm d_1}),\label{inceq}
\end{align}
where ${\bm {d_1}}=(\cos{\theta},\sin{\theta})$, $\theta \in [0,2\pi]$.
Evidently, such incident waves satisfy the homogeneous equation
\begin{align}
\Delta u^{\rm i}+k^2_0 u^{\rm i}=0.
\end{align}
The total electric field $u$ consists of the incident field $u^{\rm i}$
and the scattered field $u^{\rm s}$:
\begin{align}
u=u^{\rm i}+u^{\rm s}.
\end{align}
It follows from the equations \eqref{helmeq} and \eqref{inceq} that
the scattered field satisfies
\begin{align}
\Delta u^{\rm s}+k_0^2 (1+q)u^{\rm s}=-k_0^2 qu^{\rm i}.\label{scateq}
\end{align}
In free space, the scattered field is required to satisfy the
following Sommerfeld radiation condition
\begin{align}
\lim_{r\to\infty} \sqrt{r}\left(\frac{\partial u^{\rm s}}{\partial r}-{\rm i} k_0
u^{\rm s}\right)=0,\quad r=|x|,
\end{align}
uniformly along all directions $x/|x|$. In
practice, it is convenient to reduce the problem to a bounded
domain. For the sake of simplicity, we employ the first order
absorbing boundary condition~\cite{Ji} on the surface of the medium:
\begin{align}
\frac{\partial u^{\rm s}}{\partial n}-{\rm i} k_0 u^{\rm s}=0.\label{abc}
\end{align}

Given the incident field $u^{\rm i}$, the direct problem is to determine
the scattered field $u^{\rm s}$ for the known scatterer $q(k_0,x)$. Using
the Lax-Milgram lemma and the Fredholm alternative, the direct
problem is shown in~\cite{peijun} to have a unique solution for all
$k_0>0$. An energy estimate for the scattered field is given in this
paper, which provides a criterion for the weak scattering.
Furthermore, properties on the continuity and the Fr\'{e}chet
differentiability of the nonlinear scattering map are examined. For
the regularity analysis of the scattering map in an open domain, the
reader is referred to~\cite{B-C-M},~\cite{C-K} and~\cite{K-W}. The
inverse medium scattering problem is to determine the scatterer
$q(k_0,x)$ from the measurements on the surface of the medium,
$u^{\rm s}|_{\Gamma}$, given the incident field $u^{\rm i}$. Two
major difficulties for solving the inverse problem by optimization
methods are the ill-posedness and the presence of many local minima.
In this paper we developed a continuation method based on the
approach introduced in~\cite{peijun}. The algorithm requires
multi-frequency scattering data. Using an initial guess from the
Born approximation, each update is obtained via recursive
linearization on the wavenumber $k_0$ by solving one forward problem
and one adjoint problem of the Helmholtz equations.

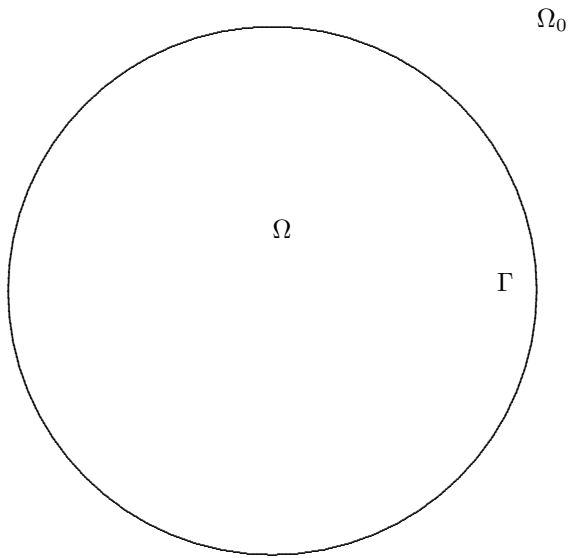
\begin{figure}
  \centering
\begin{picture}(200,190)(0,0)
  \put(90,90){\bigcircle{200}}
  \put(90,110){$\Omega$}
  \put(190,190){$\Omega_0$}
  \put(175,90){$\Gamma$}
\end{picture}
\caption{Geometry of the inverse scattering
problem}\label{inverse_geometry_simple}
\end{figure}

The plan of this paper is as follows. The analysis of the
variational problem for direct scattering is presented in section \ref{variational}. The Fr\'{e}chet differentiability of the scattering map is also
given. In section \ref{bornSec}, an initial guess of the reconstruction from
the Born approximation is derived in the case of weak scattering.
Section \ref{recursive} is devoted to numerical study of a regularized iterative
linearization algorithm. In section \ref{implementation}, we discuss the numerical implementation of the forward scattering problem and the recursive linearization algorithm. Numerical examples are presented in section \ref{example}.

\section{Analysis of the Scattering Map}\label{variational}
In this section, the direct scattering problem is studied to provide
some criterion for the weak scattering, which plays an important
role in the inversion method. The Fr\'{e}chet differentiability of
the scattering map for the problem \eqref{scateq}, \eqref{abc} is
examined.

To obtain the variational form of our boundary value problem, we multiply \eqref{scateq} by a test function $\psi$, and integrate:
\begin{align*}
\int_\Omega (\Delta u^{\rm s}\psi+k_0(1+q)u^{\rm s}\psi) dx=\int_\Omega -k_0^2 qu^{\rm i}\psi dx.
\end{align*}
Using Green's First Identity, we get
\begin{align*}
-\int_\Omega \nabla u^{\rm s}\nabla\psi dx +\int_\Gamma  \frac{\partial u^{\rm s}}{\partial n} \psi ds +
\int_\Omega k_0(1+q)u^{\rm s}\psi dx=\int_\Omega -k_0^2 qu^{\rm i}\psi dx.
\end{align*}
By \eqref{abc}, it is equivalent to
\begin{align*}
-\int_\Omega \nabla u^{\rm s}\nabla\psi dx +\int_\Gamma {\rm i}k_0 u^{\rm s} \psi ds +
\int_\Omega k_0(1+q)u^{\rm s}\psi dx=\int_\Omega -k_0^2 qu^{\rm i}\psi dx.
\end{align*}
Therefore, we introduce the bilinear form
$a: H^1(\Omega)\times H^1(\Omega)\to\mathbb{C}$
\begin{align}
a(\phi,\psi)=(\nabla \phi, \nabla \psi)-k_0^2((1+q) \phi,
\psi)-{\rm i} k_0\langle \phi, \psi\rangle,
\end{align}
and the linear functional on $H^1(\Omega)$
\begin{align}
b(\psi)=(k_0^2 q u^{\rm i}, \psi).
\end{align}
Here, we have used the standard inner products
\begin{align}
(\phi, \psi)=\int_{\Omega}
\phi\cdot\overline{\psi} {\rm d}x\,\quad \textrm{and} \quad \langle \phi,
\psi\rangle=\int_{\Gamma}
\phi\cdot\overline{\psi}{\rm d}s,
\end{align}
where the overline denotes the complex conjugate.

Then, we have the weak form of the boundary value
problem \eqref{scateq} and \eqref{abc}: find $u^{\rm s} \in H^1(\Omega)$
such that
\begin{align}
a(u^{\rm s}, \xi)=b(\xi),\quad\forall \xi\in H^1(\Omega).\label{varprob}
\end{align}

Throughout the paper, the constant $C$ stands for a positive generic
constant whose value may change step by step, but should always be
clear from the contexts.

For a given scatterer $q$ and an incident field $u^{\rm i}$, suppose a solution of the problem \eqref{scateq} and \eqref{abc} or the variational problem \eqref{varprob} $u^{\rm s} \in H^1(\Omega)$ exists, we define the
map $S(q,u^{\rm i})$ by $u^{\rm s}=S(q,u^{\rm i})$. It is easily seen that the map $S(q,u^{\rm i})$
is linear with respect to $u^{\rm i}$ but is nonlinear with respect to
$q$. Hence, we may denote $S(q,u^{\rm i})$ by $S(q)u^{\rm i}$. In the following, we will show the existence of the solution by Fredholm alternative.

Concerning the map $S(q)$, a continuity result for the map $S(q)$ is
presented in Lemma~\ref{contSq}.
\begin{lemma}
Given the scatterer $q\in L^{\infty}(\Omega)$, the direct
scattering problem \eqref{scateq} and \eqref{abc} has at most one
solution.
\end{lemma}
\begin{proof}
It is sufficient to show that $u^{\rm s}=0$ if $u^{\rm i}=0$. We begin by noting from \eqref{scateq} that
\begin{align*}
&\Im\{\int_{\Omega}(u^{\rm s}\Delta \bar{u^{\rm s}} - \bar{u^{\rm s}} \Delta u^{\rm s}  ) dV\}\\
&\quad=\Im\{\int_{\Omega}[u^{\rm s}(-k_0^2 (1+\bar{q}) \bar{u^{\rm s}}) - \bar{u^{\rm s}} (-k_0^2 (1+q) u^{\rm s}) ] dV\}\\
&\quad=\Im\{\int_{\Omega}|u^{\rm s}|^2 (-k_0^2 (1+\bar{q})+k_0^2(1+q)) dV\}\\
&\quad=\int_{\Omega} |u^{\rm s}|^2 k_0^2 (1+2 \Im(q)) dV>=0.
\end{align*}
On the other hand by the absorbing boundary condition \eqref{abc}
\begin{align*}
&\Im\{\int_{\Omega}(u^{\rm s}\Delta \bar{u^{\rm s}} - \bar{u^{\rm s}} \Delta u^{\rm s}  ) dV\}\\
&\quad=\Im\{\oint_{\Gamma}(u^{\rm s} \frac{\partial \bar{u^{\rm s}}}{\partial n} - \bar{u^{\rm s}}\frac{\partial u^{\rm s}}{\partial n}  ) dS\}\\
&\quad=\Im\{\oint_{\Gamma}(- {\rm i}k_0 u^{\rm s} \bar{u^{\rm s}} - {\rm i} k_0 \bar{u^{\rm s}} u^{\rm s}  ) dS\}\\
&\quad=-2k_0 \oint_{\Gamma}\|u^{\rm s}\|^2 dS.
\end{align*}
Hence we have $u^{\rm s}=0$ on $\Gamma$ and $\frac{\partial u^{\rm s}}{\partial n}=0$ on $\Gamma$, which implies $u^{\rm s}=0$ in $\Omega$ by a unique continuation result for Helmholtz equation~\cite{C-K}. 
\end{proof}

\begin{lemma}\label{bdnSq}
If the wavenumber $k_0$ is sufficiently small, the variational
problem \eqref{varprob} admits a unique weak solution in
$H^1(\Omega)$ and $S(q)$ is a bounded linear map from
$L^2(\Omega)$ to $H^1(\Omega)$. Furthermore, there is a constant
$C$ dependent of $\Omega$, such that
\begin{align}
\| S(q)u^{\rm i}\|_{H^1(\Omega)}\leq C
k_0^2 \|q\|_{
L^{\infty}(\Omega)}\|u^{\rm i}\|_{
L^2(\Omega)}.\label{ineqSq}
\end{align}
\end{lemma}

\begin{proof}
Decompose the bilinear form $a$ into $a=a_1+k_0^2 a_2$, where
\begin{align}
a_1(u^{\rm s}, \xi)&=(\nabla u^{\rm s}, \nabla\xi)-{\rm i} k_0\langle u^{\rm s}, \xi\rangle,\\
a_2(u^{\rm s}, \xi)&=-((1+q)u^{\rm s}, \xi).
\end{align}
We conclude that $a_1$ is coercive from
\begin{align*}
|a_1(u^{\rm s},u^{\rm s})|&=|\|\nabla u^{\rm s}\|^2_{
L^2(\Omega)}-{\rm i}
k_0\|u^{\rm s}\|^2_{L^{2}(\Gamma)}|\\
&=\sqrt{\|\nabla u^{\rm s}\|^4_{L^2(\Omega)}+k_0^2 \|u^{\rm s}\|^4_{L^{2}(\Gamma)}}\\
&\geq C \|u^{\rm s}\|^2_{H^1(\Omega)},
\end{align*}
where the last inequality may be obtained by Poincar\'e inequality. Next, we prove the compactness of
$a_2$. Define an operator $\mathcal{A}: L^2(\Omega)\to
H^1(\Omega)$ by
\begin{align}
a_1(\mathcal{A}u^{\rm s}, \xi)=a_2(u^{\rm s}, \xi),\quad\forall \xi\in
H^1(\Omega),
\end{align}
which gives
\begin{align*}
(\nabla\mathcal{A}u^{\rm s}, \nabla\xi)-{\rm i} k_0\langle\mathcal{A}u^{\rm s},
\xi\rangle =-((1+q)u^{\rm s}, \xi),\quad\forall \xi\in H^1(\Omega).
\end{align*}
Using the Lax--Milgram Lemma, it follows that
\begin{align*}
\|\mathcal{A}u^{\rm s}\|_{
H^1(\Omega)}\leq C\|u^{\rm s}\|_{
L^2(\Omega)},\label{ineq1}
\end{align*}
where the constant $C$ is independent of $k_0$. Thus $\mathcal{A}$
is bounded from $L^2(\Omega)$ to $H^1(\Omega)$ and
$H^1(\Omega)$ is compactly imbedded into $L^2(\Omega)$. Hence
$\mathcal{A}: L^2(\Omega)\to L^2(\Omega)$ is a compact operator.

Define a function $\phi \in L^2(\Omega)$ by requiring $\phi\in
H^1(\Omega)$ and satisfying
\begin{align}
a_1(\phi, \xi)=b(\xi),\quad\forall\xi\in H^1(\Omega).
\end{align}
It follows from the Lax--Milgram Lemma again that
\begin{align}
\|\phi\|_{ H^1(\Omega)}\leq C
k_0^2\|q\|_{
L^{\infty}(\Omega)}\|u^{\rm i}\|_{
L^2(\Omega)}.\label{ineq2}
\end{align}
Using the operator $\mathcal{A}$, we can see that problem \eqref{varprob} is equivalent to find $u^{\rm s}\in L^2(\Omega)$
such that
\begin{align}
(\mathcal{I}+k_0^2\mathcal{A})u^{\rm s}=\phi.\label{ineq3}
\end{align}
When the wavenumber $k_0$ is small enough, the operator
$\mathcal{I}+k_0^2\mathcal{A}$ has a uniformly bounded inverse. We
then have the estimate
\begin{align}
\|u^{\rm s}\|_{ L^2(\Omega)}\leq C\|\phi\|_{
L^2(\Omega)},\label{ineq4}
\end{align}
where the constant $C$ is independent of $k_0$.
Rearranging \eqref{ineq3}, we have $u^{\rm s}=\phi-k_0^2\mathcal{A}u^{\rm s}$, so
$u^{\rm s} \in H^1(\Omega)$ and, by the estimate \eqref{ineq1} for the
operator $\mathcal{A}$, we have
\begin{align*}
\|u^{\rm s}\|_{ H^1(\Omega)}\leq \| \phi\|_{
H^1(\Omega)}+ C k_0^2 \|u^{\rm s}\|_{
L^2(\Omega)}.
\end{align*}
The proof is complete by combining the estimates \eqref{ineq4}
and \eqref{ineq2} and observing that $u^{\rm s}=S(q)u^{\rm i}$. \qedhere
\end{proof}

For a general wavenumber $k_0>0$, from the equation \eqref{ineq3},
the existence follows from the Fredholm alternative and the
uniqueness result. 

\begin{rema} \label{rmkweak}
It follows from the explicit form of the incident
field \eqref{inceq} and the estimate \eqref{ineqSq} that
\begin{align}
\|u^{\rm s}\|_{ H^1(\Omega)}\leq C k_0^2
|\Omega|^{\frac{1}{2}}\|q\|_{ L^{\infty}(\Omega)},
\end{align}
where the
constant $C$ depends on $\Omega$.
\end{rema}

\begin{lemma}\label{contSq}
Assume that $q_1, q_2\in L^{\infty}(\Omega)$. Then
\begin{align}
\| S(q_1)u^{\rm i}-S(q_2)u^{\rm i}\|_{ H^1(\Omega)}\leq C\|
q_1-q_2\|_{ L^{\infty}(\Omega)}\|u^{\rm i}\|_{
L^2(\Omega)},\label{contSqIneq}
\end{align}
where the constant $C$ depends on $k_0, \Omega$, and $\|
q_2\|_{ L^{\infty}(\Omega)}$.
\end{lemma}
\begin{proof}
Let $u^{\rm s}_1=S(q_1)u^{\rm i}$ and $u^{\rm s}_2=S(q_2)u^{\rm i}$. It follows that for
$j=1, 2$
\begin{align*}
\Delta u^{\rm s}_j+k_0^2(1+q_j)u^{\rm s}_j=-k_0^2 q_j u^{\rm i}.
\end{align*}
By setting $w=u^{\rm s}_1-u^{\rm s}_2$, we have
\begin{align*}
\Delta w+k_0^2(1+q_1)w=-k_0^2(q_1-q_2)(u^{\rm i}+u^{\rm s}_2).
\end{align*}
The function $w$ also satisfies the boundary condition \eqref{abc}.

We repeat the procedure in the proof of Lemma~\ref{bdnSq} 

\begin{align*}
\| w\|_{ H^1(\Omega)}\leq C\|
q_1-q_2\|_{
L^{\infty}(\Omega)}\|u^{\rm i}+u^{\rm s}_2\|_{ L^2(\Omega)}.
\end{align*}
Using Lemma~\ref{bdnSq} again for $u^{\rm s}_2$ yields
\begin{align*}
\|u^{\rm s}_2\|_{ H^1(\Omega)}\leq C\| q_2\|_{
L^{\infty}(\Omega)}\|u^{i}\|_{ L^2(\Omega)},
\end{align*}
which gives
\begin{align*}
\| S(q_1)u^{\rm i}-S(q_2)u^{\rm i}\|_{ H^1(\Omega)}\leq C\|
q_1-q_2\|_{L^{\infty}(\Omega)}\|u^{\rm i}\|_{
L^2(\Omega)},
\end{align*}
where the constant $C$ depends on $\Omega, k_0$, and $\|
q_2\|_{ L^{\infty}(\Omega)}$. \qedhere
\end{proof}

Let $\gamma$ be the restriction (trace) operator to the boundary
$\Gamma$. By the trace theorem, $\gamma$ is a bounded linear
operator from $H^1(\Omega)$ onto $H^{\frac{1}{2}}(\Gamma)$. We can now
define the scattering map $M(q)=\gamma S(q).$

Next, consider the Fr\'{e}chet differentiability of the scattering
map. Recall the map $S(q)$ is nonlinear with respect to $q$.
Formally, by using the first order perturbation theory, we obtain
the linearized scattering problem of \eqref{scateq}, \eqref{abc}
with respect to a reference scatterer $q$,
\begin{align}\label{ptb}
&\Delta v+k_0^2(1+q)v=-k_0^2\delta q(u^{\rm i}+u^{\rm s}),\\
&\frac{\partial v}{ \partial n}-{\rm i} k_0 v=0,\nonumber
\end{align}
where $u^{\rm s}=S(q)u^{\rm i}$.

Define the formal linearzation $T(q)(\delta q)$ of the map $S(q)u^{\rm i}$ by
$v=T(q)(\delta q, u^{\rm i})$, where $v$ is the solution of the
problem \eqref{ptb}. The following is a boundedness
result for the map $T(q)$. A proof may be given by following step by
step the proofs of Lemma ~\ref{bdnSq}. Hence we omit it here.
\begin{lemma}\label{bndTq}
Assume that $q, \delta q\in L^{\infty}(\Omega)$ and $u^{\rm i}$ is the
incident field. Then $v=T(q)(\delta q, u^{\rm i})\in H^1(\Omega)$ with
the estimate
\begin{align}
\| T(q)(\delta q, u^{\rm i})\|_{ H^1(\Omega)}\leq C\|\delta
q\|_{ L^{\infty}(\Omega)}\|u^{\rm i}\|_{
L^2(\Omega)},\label{ineqTq}
\end{align}
where the constant $C$ depends on $k_0, \Omega$, and $\|
q\|_{ L^{\infty}(\Omega)}$.
\end{lemma}
The next lemma is concerned with the continuity property of the map.
\begin{lemma}
For any $q_1, q_2\in L^{\infty}(\Omega)$ and an incident field
$u^{\rm i}$, the following estimate holds

\begin{align}
\| T(q_1)(\delta q, u^{\rm i})-T(q_2)(\delta q, u^{\rm i})\|_{
H^1(\Omega)}\leq C\| q_1-q_2\|_{
L^{\infty}(\Omega)}\|\delta q\|_{
L^{\infty}(\Omega)}\|u^{\rm i}\|_{
L^2(\Omega)},\label{ineqcontTq}
\end{align}

where the constant $C$ depends on $k_0, \Omega$, and $\|
q_2\|_{ L^{\infty}(\Omega)}$.
\end{lemma}

\begin{proof}
Let $v_j=T(q_j)(\delta q, u^{\rm i})$, for $j=1, 2$. It is easy to see
that
\begin{align*}
\Delta (v_1-v_2)&+k_0^2(1+q_1)(v_1-v_2)=-k_0^2\delta q(u^{\rm s}_1-u^{\rm s}_2)-k_0^2(q_1-q_2)v_2,
\end{align*}
where $u^{\rm s}_j=S(q_j)u^{\rm i}$.

Similar to the proof of Lemma~\ref{bdnSq}, we get
\begin{align*}
\| v_1-v_2\|_{ H^1(\Omega)}\leq C(\|\delta
q\|_{
L^{\infty}(\Omega)}\|u^{\rm s}_1-u^{s}_2\|_{
H^1(\Omega)}+\| q_1-q_2\|_{ L^{\infty}(\Omega)}\|
v_2\|_{ H^1(\Omega)}).
\end{align*}
From Lemma~\ref{contSq} and Lemma~\ref{bndTq}, we obtain
\begin{align*}
\| v_1-v_2\|_{ H^1(\Omega)}\leq C\|
q_1-q_2\|_{ L^{\infty}(\Omega)}\|\delta
q\|_{ L^{\infty}(\Omega)}\|u^{\rm i}\|_{
L^2(\Omega)},
\end{align*}
which completes the proof. \qedhere
\end{proof}

The following result concerns the differentiability property of
$S(q)$.

\begin{lemma}
Assume that $q, \delta q\in L^{\infty}(\Omega)$. Then there is a
constant $C$ dependent of $k_0, \Omega$, and $\| q\|_{
L^{\infty}(\Omega)}$, for which the following estimate holds
\begin{align}
\| S(q+\delta q)u^{\rm i}-S(q)u^{\rm i}-T(q)(\delta q, u^{\rm i})\|_{
H^1(\Omega)}\leq C\|\delta q\|^2_{
L^{\infty}(\Omega)}\|u^{\rm i}\|_{
L^2(\Omega)}.\label{ineqdifSq}
\end{align}
\end{lemma}

\begin{proof}
By setting $u^{\rm s}_1=S(q)u^{\rm i}, u^{\rm s}_2=S(q+\delta q)u^{\rm i}$, and
$v=T(q)(\delta q, u^{\rm i})$, we have
\begin{align}
&\Delta u_1^{\rm s}+k_0^2(1+q)u^{\rm s}_1=-k_0^2 q u^{\rm i},\\
&\Delta u_2^{\rm s}+k_0^2(1+q+\delta q)u^{\rm s}_2=-k_0^2(q+\delta q)u^{\rm i},\\
&\Delta v+k_0^2(1+q)v=-k_0^2 \delta q u^{\rm s}_1-k_0^2\delta q u^{\rm i}.
\end{align}
In addition, $u^{\rm s}_1, u^{\rm s}_2$, and $v$ satisfy the boundary
condition \eqref{abc}.

Denote $U=u^{\rm s}_2-u^{\rm s}_1-v$. Then
\begin{align}
\Delta U+k_0^2(1+q)U=-k_0^2\delta q(u^{\rm s}_2-u^{\rm s}_1).
\end{align}
Similar arguments as in the proof of Lemma~\ref{bdnSq} give
\begin{align*}
\| U\|_{ H^1(\Omega)}\leq C\|\delta
q\|_{
L^{\infty}(\Omega)}\|u^{\rm s}_2-u^{\rm s}_1\|_{ H^1(\Omega)}.
\end{align*}
From Lemma ~\ref{bdnSq}, we obtain further that
\begin{align*}
\| U\|_{ H^1(\Omega)}\leq C\|\delta
q\|^2_{ L^{\infty}(\Omega)}\|u^{\rm i}\|_{
L^2(\Omega)}.
\end{align*} \qedhere
\end{proof}
Finally, by combining the above lemmas, we arrive at
\begin{thm}
The scattering map $M(q)$ is Fr\'{e}chet differentiable with respect
to $q$ and its Fr\'{e}chet derivative is
\begin{align}
{\rm D}M(q)=\gamma T(q).\label{frechet}
\end{align}
\end{thm}

\section{Inverse Medium Scattering}
In this section, a regularized recursive linearization method for
solving the inverse scattering problem of the Helmholtz equation in
two dimensions is proposed. The algorithm requires multi-frequency
Dirichlet and Neumann scattering data, and the recursive
linearization is obtained by a continuation method on the wavenumber
$k_0$. It first solves a linear equation (Born approximation) at the
lowest $k_0$, which gives the initial guess of $q(k_0,x)$. Updates
are subsequently obtained by using a sequence of increasing
wavenumbers. For each iteration, one forward and one adjoint
equation are solved. 

Let $\Omega$ be the circle that contains the underlying medium; let
$\Gamma=\partial \Omega$ be the surface; let
$\Omega_0=\mathbb{R}^2/\bar{\Omega}$.
The inverse problem can be stated as follows. Given $u^{\rm s}(x; k )$ for all $x\in\Gamma$,
and all $k > 0$, find the function $q (x)$, $x\in\Omega$, assuming that conditions \eqref{helmb}-\eqref{boundarycond} hold:
\begin{align}
&\Delta u+k_0^2(1+q) u=0 \quad & \textrm{in } \Omega,\label{helmb}\\
&\Delta u^0+k_0^2 u^0=0 \quad & \textrm{in } \Omega_0,\label{helm0}\\
&u=u^0=u^{\rm s}+u^{\rm i} \quad & \textrm{on } \Gamma,\label{jump1}\\
&\frac{\partial u}{\partial n}=\frac{\partial
u^{\rm s}}{\partial n}+\frac{\partial u^{\rm i}}{\partial n} \quad & \textrm{on } \Gamma,\label{jump2}\\
&\frac{\partial u^{\rm s}}{\partial n}-{\rm i} k_0 u^{\rm s}=0 \quad & \textrm{on }
\Gamma,\label{boundarycond}
\end{align}

where \eqref{jump1} and \eqref{jump2} are the jump conditions on the
surface of the medium.

\subsection{Born Approximation}\label{bornSec}
Define a test function $\hat{u}=\exp( {\rm i} k_0
{\bm x}\cdot {\bm d_2}), {\bm d_2}=(\cos\theta, \sin\theta), \theta\in [0,
2\pi]$. Hence $\hat{u}$ satisfies:
\begin{align}
\Delta \hat{u}+k_0^2\hat{u}=0 \quad \textrm{in } \Omega.\label{uref}
\end{align}

Multiplying the equation \eqref{helmb} by $\hat{u}$, and integrating
over $\Omega$ on both sides, we have
\begin{align*}
\int_{ \Omega}\hat{u} \Delta u
{\rm d}x+k_0^2\int_{ \Omega} (1+q)\hat{u} u {\rm d}x=0.
\end{align*}
Integration by parts yields
\begin{align*}
\int_{ \Omega} \Delta \hat{u}u {\rm d}x+\int_{
\Omega}k_0^2 (1+q)u\hat{u}{\rm d}x+\int_{
\Gamma}(\hat{u}\frac{\partial u}{\partial n}-u\frac{\partial
\hat{u}}{\partial n}){\rm d}s=0.
\end{align*}
We have by noting \eqref{uref} and the jump conditions \eqref{jump1}
and \eqref{jump2} that
\begin{align*}
\int_{ \Omega}k_0^2 q u\hat{u}{\rm d}x=\int_{
\Gamma} (u^{\rm s} \frac{\partial\hat{u}}{\partial n}-
\hat{u}\frac{\partial u^{\rm s}}{\partial n}){\rm d}s+\int_{
\Gamma} (u^{\rm i} \frac{\partial\hat{u}}{\partial n}-
\hat{u}\frac{\partial u^{\rm i}}{\partial n}){\rm d}s,
\end{align*}
where we take into account that $q$ has compact support in $\Omega$.
Using the special form of the incident wave and the test function,
we then get
\begin{align}\label{born}
\begin{split}
\int_{ \Omega}k_0^2 q u& \exp( {\rm i} k_0
{\bm x}\cdot {\bm d_1}){\rm d}x\\
&=\int_{ \Gamma}({\rm i} k_0
{\bm n}\cdot {\bm d_1}\exp( {\rm i} k_0 {\bm x}\cdot
{\bm d_1})u^{\rm s}-\exp({\rm i} k_0 {\bm x}\cdot {\bm d_1})\frac{\partial
u^{\rm s}}{\partial n}){\rm d}s\\
&\quad+\int_{
\Gamma}\exp( {\rm i} k_0 {\bm x}\cdot{\bm d_2}+{\rm i} k_0
{\bm x}\cdot {\bm d_1})({\rm i} k_0 {\bm n}\cdot {\bm d_1}-{\rm i} k_0 {\bm n}\cdot
{\bm d_2}){\rm d}s.
\end{split}
\end{align}

From Lemma~\ref{bdnSq} and Remark~\ref{rmkweak}, for a small
wavenumber, the scattered field is weak comparing to the incident field, so $u=u^{\rm s}+u^{\rm i}\approx u^{\rm i}$. 
We drop the nonlinear term of \eqref{born} and obtain the linearized integral equation
\begin{align}\label{borndropped}
\begin{split}
\int_{ \Omega}k_0^2q_0&(x) \exp(
{\rm i} k_0{\bm x}\cdot{\bm d_2}+{\rm i} k_0{\bm x}\cdot{\bm d_1}){\rm d}x\\
&=\int_{
\Gamma}({\rm i} k_0 {\bm n}\cdot{\bm d_1}\exp( {\rm i} k_0 {\bm x}\cdot
{\bm d_1})u^{\rm s}-\exp( {\rm i} k_0 {\bm x}\cdot {\bm d_1})\frac{\partial
u^{\rm s}}{\partial n}){\rm d}s\\
&\quad+\int_{ \Gamma}\exp( {\rm i} k_0
{\bm x}\cdot{\bm d_2}+{\rm i} k_0 {\bm x}\cdot{\bm d_1})({\rm i} k_0 {\bm n}\cdot {\bm d_1}-
{\rm i} k_0 {\bm n}\cdot {\bm d_2}){\rm d}s,
\end{split}
\end{align}
which is the Born approximation.

Since the scatterer $q_0(k_0,x)$ has a compact support, we use the
notation
\begin{align}
\hat{q_0}(\xi)=\int_{\Omega} q_0 (x)\exp(
{\rm i} k_0 {\bm x}\cdot{\bm d_2}+{\rm i} k_0 {\bm x}\cdot{\bm d_1}){\rm d}x,
\end{align}
where $\hat{q}_0(\xi)$ is the Fourier transform of $q_0(x)$ with
$\xi=k_0({\bm d_1}+{\bm d_2})$. Choose
\begin{align}
{\bm d_j}=\left(\cos\theta_j, \sin\theta_j\right),\quad j=1,2,
\end{align}
where $\theta_j$ are spherical angles. It is obvious that the domain
$[0, 2\pi]$ of $\theta_j, j=1, 2$, corresponds to the ball
$\{\xi\in\mathbb{R}^2: |\xi|\leq 2 k_0\}$. Thus, the Fourier
modes of $\hat{q}_0 (\xi)$ in the ball $\{\xi: |\xi|\leq
2 k_0\}$ can be determined. The scattering data with the higher
wavenumber must be used in order to recover more modes of the true
scatterer.

The integral equation \eqref{borndropped} can be written as the
operator form
\begin{align}
A(k_0, \theta; x)q(x)=f(k_0,\theta).\label{bornq}
\end{align}
Thus
\begin{align}
q_0(k_0,x)=(A*A+\alpha I)^{-1}f,
\end{align}
where $\alpha$ is a relaxation parameter. $q_0$ is used as the starting point of the following recursive
linearization algorithm.

\subsection{Recursive Linearization}\label{recursive}
As discussed in the previous section, when the wavenumber is small,
the Born approximation allows a reconstruction of those Fourier
modes less than or equal to $2 k_0$ for the function $q(x)$. We
now describe a procedure that recursively determines
$q_{k_0}$ at $k_0=k_j$ for $j=1,2,...$ with the
increasing wavenumbers. Suppose now that the scatterer
$q_{\tilde{k}}$ has been recovered at some wavenumber $\tilde{k}$,
and that the wavenumber $k$ is slightly larger than $\tilde{k}$. We
wish to determine $q_k$, or equivalently, to determine the
perturbation
\begin{align}
\delta q=q_{k}-q_{\tilde{k}}.
\end{align}

For the reconstructed scatterer $q_{\tilde{k}}$, we solve at the
wavenumber $k$ the forward scattering problem
\begin{align}\label{forwardproblem2}
&\Delta \tilde{u}+k^2(1+q_{\tilde{k}}(k,x))\tilde{u}=0 \quad & \textrm{in } \Omega,\\
&\Delta \tilde{u}^0+k^2 \tilde{u}^0=0 & \textrm{in }
\Omega_0,\nonumber\\
&\tilde{u}=\tilde{u}^0=\tilde{u}^{\rm s}+u^{\rm i} & \textrm{on } \Gamma,\nonumber\\
&\frac{\partial \tilde{u}}{\partial n}=\frac{\partial
\tilde{u}^{\rm s}}{\partial n}+\frac{\partial u^{\rm i}}{\partial n} & \textrm{on } \Gamma,\nonumber\\
&\frac{\partial \tilde{u}^{\rm s}}{\partial n}-{\rm i} k \tilde{u}^{\rm s}=0 &
\textrm{on } \Gamma .\nonumber
\end{align}
For the scatterer $q_k$, we have
\begin{align}\label{forwardproblem1}
&\Delta u+k^2(1+q_k(k,x))u=0 \quad & \textrm{in } \Omega,\\
&\Delta u^0+k^2 u^0=0 & \textrm{in } \Omega_0,\nonumber\\
&u=u^0=u^{\rm s}+u^{\rm i} & \textrm{on } \Gamma,\nonumber\\
&\frac{\partial u}{\partial n}=\frac{\partial
u^{\rm s}}{\partial n}+\frac{\partial u^{\rm i}}{\partial n} & \textrm{on } \Gamma,\nonumber\\
&\frac{\partial u^{\rm s}}{\partial n}-{\rm i} k u^{\rm s}=0 & \textrm{on }
\Gamma.\nonumber
\end{align}
Subtracting \eqref{forwardproblem2} from \eqref{forwardproblem1} and
omitting the second-order smallness in $\delta q$ and in $\delta
u=u-\tilde{u}$, we obtain
\begin{align}\label{deltaproblem}
&\Delta \delta u+k^2(1+q_{\tilde{k}}(k,x))\delta u=-k^2 \delta q \tilde{u} \quad & \textrm{in } \Omega,\\
&\Delta \delta u^0+k^2 \delta u^0=0 & \textrm{in }
\Omega_0,\nonumber\\
&\delta u=\delta u^0=\delta u^{\rm s} &\textrm{on } \Gamma,\nonumber\\
&\frac{\partial \delta u}{\partial n}= \frac{\partial
\delta u^{\rm s}}{\partial n} & \textrm{on } \Gamma,\nonumber\\
&\frac{\partial \delta u^{\rm s}}{\partial n}-{\rm i} k \delta u^{\rm s}=0 & \textrm{on
} \Gamma.\nonumber
\end{align}

For the scatterer $q_k$, we define the scattering map
\begin{align}
M(q_k)=u^0|_{ \Gamma},
\end{align}
where $u^0$ is the total field data corresponding to the incident
wave $u^{\rm i}$. Let ${\rm D}M(q_{k})$ be the Fr\'{e}chet derivative of $M(q_k)$ and denote
the residual operator by
\begin{align}
R(q_{\tilde{k}})=u^0|_{
\Gamma}-\tilde{u}^0|_{ \Gamma}.
\end{align}
It follows from \eqref{frechet} that
\begin{align}
{\rm D}M(q_{\tilde{k}})\delta q=R(q_{\tilde{k}}).\label{frecheteq}
\end{align}
In order to reduce the computation cost and instability, we consider
the one-step Landweber iteration of \eqref{frecheteq}, which has the form
\begin{align}
\delta q=\beta {\rm D}M^*(q_{\tilde{k}})R(q_{\tilde{k}}),\label{landweber}
\end{align}
where $\beta$ is a relaxation parameter and ${\rm D}M^*(q_{\tilde{k}})$ is
the adjoint operator of ${\rm D}M(q_{\tilde{k}})$.

In order to compute the correction $\delta q$, we need some
efficient way to compute ${\rm D}M^*(q_{\tilde{k}}) R(q_{\tilde{k}})$,
which is given by the following theorem.
\begin{thm}
Given residual $R(q_{\tilde{k}})$, the adjoint Fr\'{e}chet derivative ${\rm D}M^*(q_{\tilde{k}})$
satisfies
\begin{align}
[{\rm D}M^*(q_{\tilde{k}}) R(q_{\tilde{k}})
](x)=\bar{\tilde{u}}(x)\cdot\psi,\label{landwebereq}
\end{align}
where $\tilde{u}$ is the solution of \eqref{forwardproblem2}, $\psi$ is the solution of the
adjoint problem of \eqref{deltaproblem}.
\end{thm}
\begin{proof}
Let $\tilde{u}$ be the solution of \eqref{forwardproblem2}, and $\delta u$ be the
solution of \eqref{deltaproblem}.

Multiplying the first equation in \eqref{deltaproblem} with the complex
conjugate of a test function $\psi$ and integrating over $\Omega$ on both sides,
we obtain
\begin{align*}
\int_{ \Omega}\Delta \delta u
\bar{\psi}{\rm d}x+\int_{ \Omega}
k^2(1+q_{\tilde{k}})\delta u \bar{\psi}{\rm d}x=\int_{
\Omega}-k^2 \delta q \tilde{u} \bar{\psi}{\rm d}x.
\end{align*}
Green's second identity yields
\begin{align*}
\int_{ \Omega} \delta u
\Delta \bar{\psi}{\rm d}x+\int_{ \Gamma}(\bar{\psi} \frac{\partial \delta
u}{\partial n}-\delta u\frac{\partial \bar{\psi}}{\partial n}){\rm d}s
+\int_{ \Omega}
k^2(1+q_{\tilde{k}})\delta u \bar{\psi}{\rm d}x
=-k^2 \int_{ \Omega} \delta q
\tilde{u}\bar{\psi}{\rm d}x.
\end{align*}

We now define the adjoint problem
\begin{align}\label{bcres}
&\Delta \psi+\bar{k}^2(1+\bar{q}_{\tilde{k}}(k,x))\psi=0 \quad & \textrm{in } \Omega,\\
&\Delta \psi+k^2 \psi=0 & \textrm{in } \Omega_0,\nonumber\\
&\frac{\partial \psi}{\partial n}+{\rm i} \bar{k}
\psi=(u^0-\tilde{u}^0)\bar{k}^2 & \textrm{on } \Gamma.\nonumber
\end{align}
Since the existence and uniqueness of the weak solution for the
adjoint problem may be established by following the same proof of
Lemma~\ref{bdnSq}, we omit the proof here.

The above equation is then reduced to
\begin{align*}
\int_{ \Gamma}(\bar{\psi} \frac{\partial \delta
u}{\partial n}-\delta u\frac{\partial \bar{\psi}}{\partial n}){\rm d}s
=-k^2 \int_{ \Omega} \delta q
\tilde{u}\bar{\psi}{\rm d}x.
\end{align*}
It follows from boundary conditions of \eqref{deltaproblem} and the adjoint problem that
\begin{align*}
\int_{ \Gamma}\overline{(u^0-\tilde{u}^0)}k^2 \delta
u {\rm d}s =k^2 \int_{ \Omega} \delta q \tilde{u}\bar{\psi}{\rm d}x,
\end{align*}
Extending the integration domain to $\Omega$ and using the operators defined above, we obtain 
\begin{align*}
\int_{ \Omega}{\rm D}M(q_{\tilde{k}})\delta q
\overline{R(q_{\tilde{k}})} {\rm d}x =\int_{ \Omega} \delta q
\tilde{u}\bar{\psi}{\rm d}x.
\end{align*}
We know from the adjoint operator ${\rm D}M^*(q_{\tilde{k}})$ that
\begin{align*}
\int_{ \Gamma}\delta q 
\overline{{\rm D}M^*(q_{\tilde{k}})R(q_{\tilde{k}})} {\rm d}s =\int_{ \Omega}
\delta q \tilde{u}\bar{\psi}{\rm d}x.
\end{align*}
Since it holds for any $\delta q$, we have
\begin{align*}
\overline{{\rm D}M^*(q_{\tilde{k}})R(q_{\tilde{k}})}=
\tilde{u}\bar{\psi}.
\end{align*}
Taking the complex conjugate of the above equation yields the result. \qedhere
\end{proof}

Using this theorem, we can rewrite \eqref{landweber} as
\begin{align}
\delta
q=\beta\bar{\tilde{u}}\psi.\label{deltaq}
\end{align}

So for each incident wave and each wavenumber $k_0$, we have to
solve one forward problem \eqref{forwardproblem2} along with one
adjoint problem \eqref{bcres}. Since the adjoint problem has a
similar variational form as the forward problem, essentially, we
need to compute two forward problems at each sweep. Once $\delta q$
is determined, $q_{\tilde{k}}$ is updated by $q_{\tilde{k}}+\delta
q$.

\begin{figure}
\begin{center}
\subfigure[]{
\resizebox*{6cm}{!}{
\includegraphics{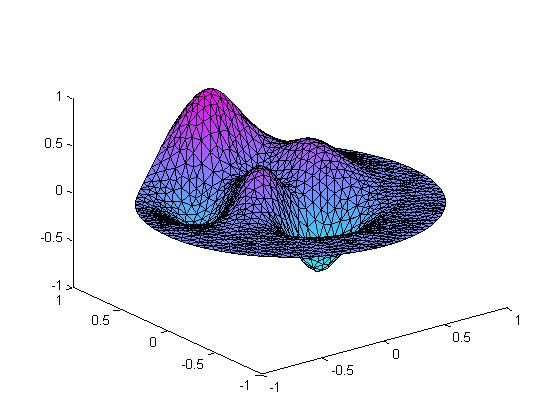}}}
\subfigure[]{
\resizebox*{6cm}{!}{
\includegraphics{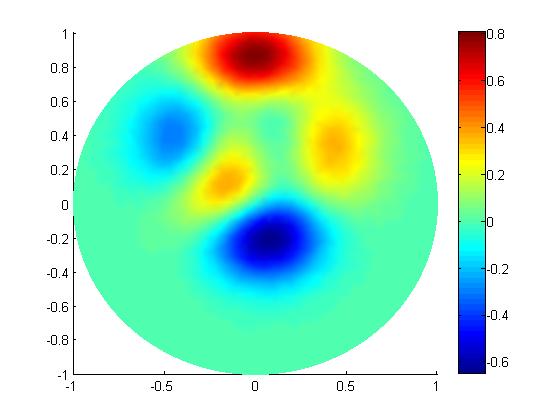}}}
\subfigure[]{
\resizebox*{6cm}{!}{
\includegraphics{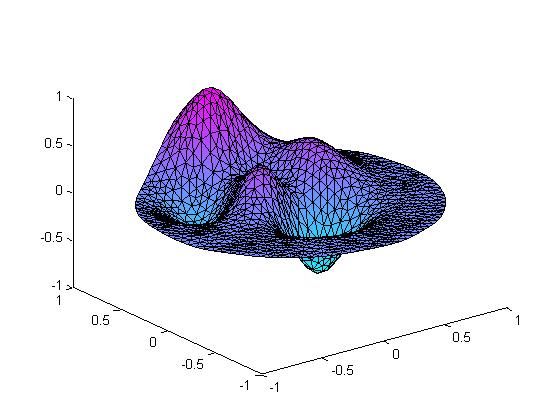}}}
\subfigure[]{
\resizebox*{6cm}{!}{
\includegraphics{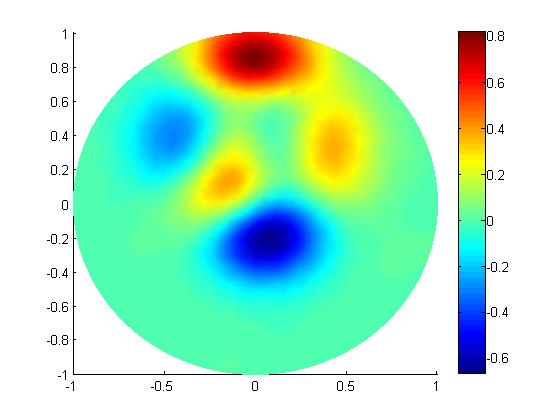}}}
\subfigure[]{
\resizebox*{6.5cm}{!}{
\includegraphics{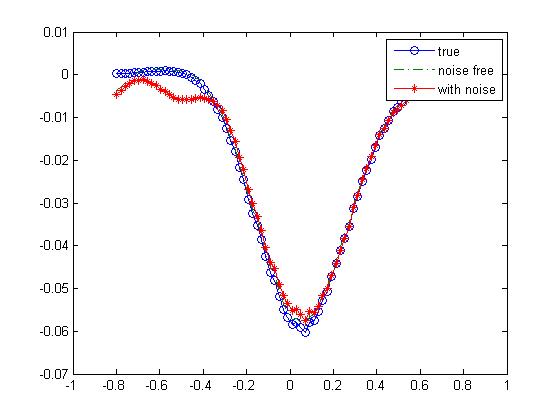}}}
\end{center}
\caption{Example 1: (a,b): surface and contour views of the true scatterer function; (c,d): final reconstruction of the scatterer function with noisy data. (e): comparison of the true scatterer, final reconstruction without noise and final reconstruction with noisy data at cross section $y=-0.6$.}\label{Example1}
\end{figure}

\section{Implementation}\label{implementation}
In this section, we discuss the numerical solution of the forward
scattering problem and the computational issues of the recursive
linearization algorithm.

The scattering data are obtained by numerical solution of the
forward scattering problem. To implement the algorithm numerically,
we employ Nystr\"{o}m's method in the exterior region
$\Omega_0$ and add some suitable boundary conditions on
$\Gamma$. Readers are referred to~\cite{Kress} for a
detailed description of Nystr\"{o}m's method. See also~\cite{C-K}
for the implementation of Nystr\"{o}m's method on integral equations
generated by Helmholtz equation. Based on Kirsch and Monk's idea
in~\cite{K-M1} and~\cite{K-M2}, the exterior problem is solved by
integral equation with radiation condition.

Define the space
\begin{align}
W(\Omega_0)=\{u^{\rm s} \in
H^1_{\mathrm{loc}}(\Omega_0)|\quad
\lim_{r\to\infty} \sqrt{r}\left(\frac{\partial
u^{\rm s}}{\partial r}-{\rm i} k_0 u^{\rm s}\right)=0,\quad r=|x|\}.
\end{align}
Define the operators
\begin{align}
&G_{\rm i}: \quad H^{\frac{1}{2}}(\Gamma)\rightarrow H^1(\Omega),\\
&G_{\rm e}: \quad H^{\frac{1}{2}}(\Gamma)\rightarrow
W(\Omega_0)
\end{align}
by the following boundary problems. Given $\lambda_{\Gamma} \in
H^{\frac{1}{2}}(\Gamma)$, define $G_{\rm i} \lambda_{\Gamma}=w$ where
$w\in H^1(\Omega)$ is the weak solution of
\begin{align}
&\Delta w+k^2(1+q_{\tilde{k}}(k,x))w=0 \quad & \textrm{in } \Omega,\nonumber\\
&\frac{\partial w}{\partial n}+{\rm i} k
w=\lambda_{\Gamma} & \textrm{on } \Gamma.
\end{align}
Similarly define $G_{\rm e} \lambda_{\Gamma}=w$ as the
weak solution of
\begin{align}
&\Delta w+k^2w=0 & \textrm{in } \Omega_0,\nonumber\\
&\frac{\partial w}{\partial n}+{\rm i} k w=\lambda_{
\Gamma} &
\textrm{on } \Gamma,\nonumber\\
&\lim_{ r\to\infty} \sqrt{r}\left(\frac{\partial
w}{\partial r}-{\rm i} k w\right)=0,\quad r=|x|.
\end{align}

To ensure continuity of solution of the forward problem across
$\Gamma$, it suffices to choose $\lambda_{\Gamma} \in
H^{\frac{1}{2}}(\Gamma)$ such that
\begin{align}
&G_{\rm i} \lambda_{\Gamma}+G_{\rm i}(\frac{\partial u^{\rm i}}{\partial n}+{\rm i} k_0
u^{\rm i})=G_{\rm e} \lambda_{ \Gamma}+u^{\rm i}
\quad & \textrm{on } \Gamma.\label{eqbc1}
\end{align}
The function $\lambda_{\Gamma}$ is approximated by trigonometric polynomials of order
$N$. Represent $\Gamma$ by $x(t)=(r_{\Gamma}\cos\theta,
r_{\Gamma}\sin\theta), 0\leq\theta\leq 2\pi$. Write
$\lambda_{\Gamma}=\sum_{n=-N}^{N-1} a_n
\exp({\rm i}n\theta_1)$. Thus
for \eqref{eqbc1}, $2N+1$ finite element problems need to be solved
on the left hand side, and
since $\Gamma$ is a circle, we can compute
$G_{\rm e}\lambda_{ \Gamma}$ explicitly as a finite linear
combination of Hankel functions:
\begin{align}
G_{\rm e}\lambda_{
\Gamma}=\frac{1}{k_0}\sum_{n=-N}^{N-1}\frac{a_n
\exp({\rm i}n\theta)}{(H_n^{(1)})'(k_0 r_{
\Gamma})+{\rm i} H_n^{(1)}(k_0 r_{ \Gamma})}H^{(1)}_n(k_0
r)
\end{align}
As for the adjoint problem, the continuity conditions are:
\begin{align}
&G_{\rm i} \lambda_{\Gamma}+G_{\rm i}((\overline{u-\tilde{u}})k_{\rm b}^2)=G_{\rm e}
\lambda_{ \Gamma}-G_{\rm e}((\overline{u-\tilde{u}})k^2)
\quad \textrm{on } \Gamma.
\end{align}

\begin{figure}
\begin{center}
\subfigure[]{
\resizebox*{6cm}{!}{
\includegraphics{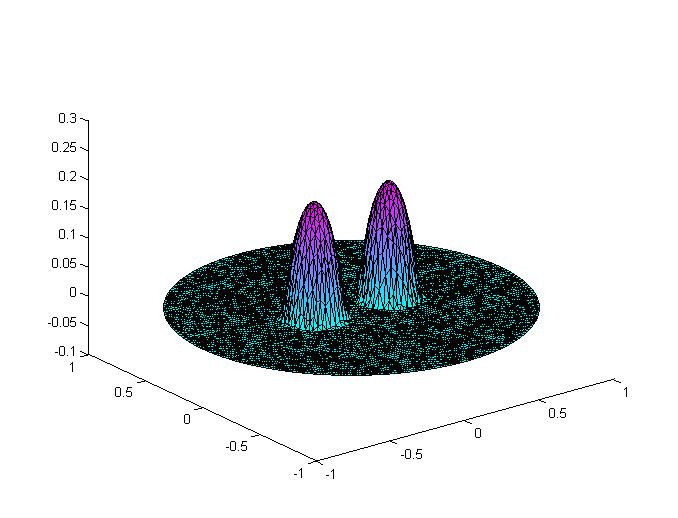}}}
\subfigure[]{
\resizebox*{6cm}{!}{
\includegraphics{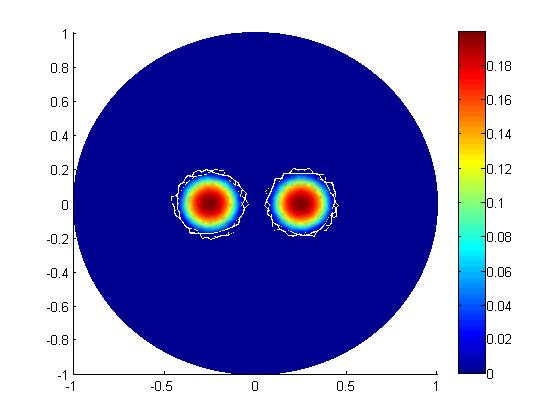}}}
\subfigure[]{
\resizebox*{6cm}{!}{
\includegraphics{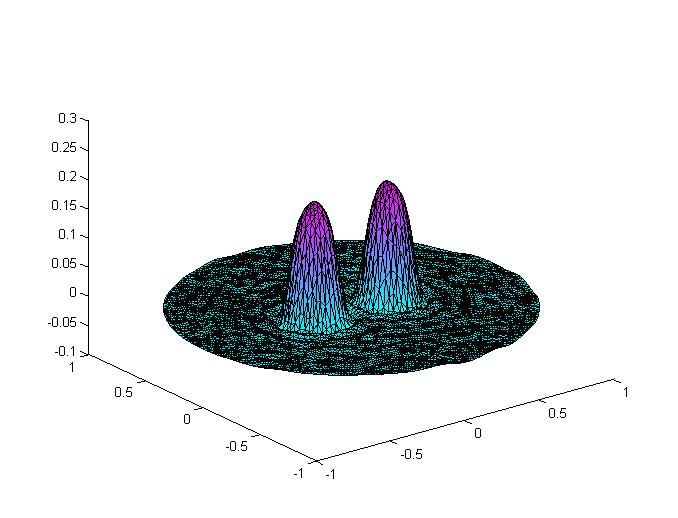}}}
\subfigure[]{
\resizebox*{6cm}{!}{
\includegraphics{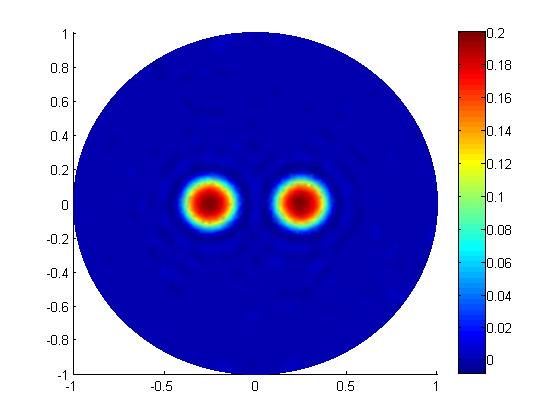}}}
\subfigure[]{
\resizebox*{6cm}{!}{
\includegraphics{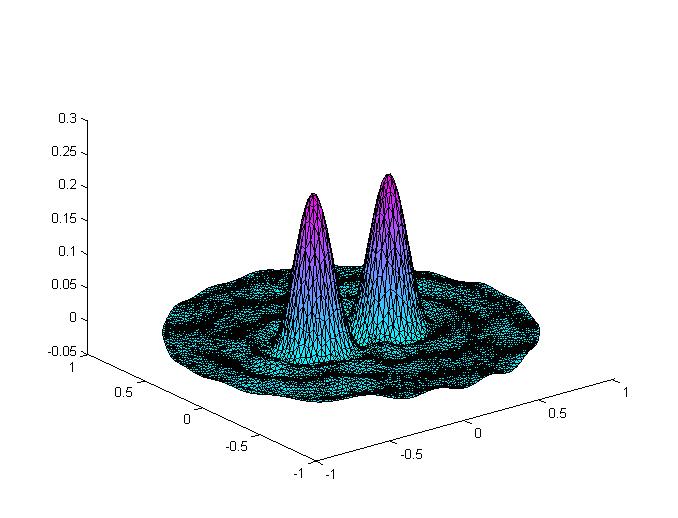}}}
\subfigure[]{
\resizebox*{6cm}{!}{
\includegraphics{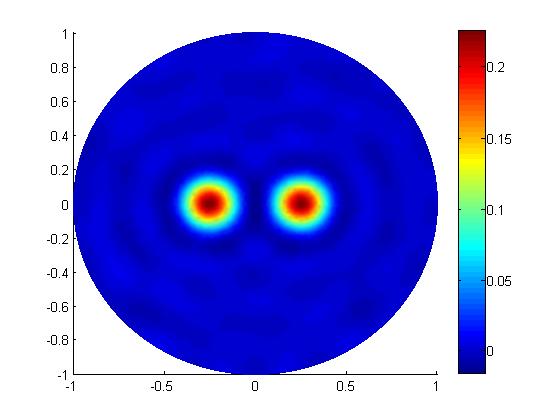}}}
\end{center}
\caption{Example 2: (a,b): surface and contour views of the true scatterer function; (c,d): Final reconstruction of the scatterer function with noise-free data; (e,f): final reconstruction of the scatterer function with noisy data. } \label{Example2}
\end{figure}

\section{Numerical Experiments}\label{example}
In the following, to illustrate the performance of the algorithm,
two numerical examples are presented for reconstructing the
scatterer of the Helmholtz equation in two dimensions. We performed four reconstructions, each example with noise-free and noisy data that contains $2\%$ multiplicative noise. The scattered field takes the form $u^{\rm s}\vert_\Gamma :=(1+0.02\textrm{ rand})u^{\rm s}\vert_\Gamma$, where rand gives uniformly distributed random numbers in $[-1,1]$.

\begin{example} Let
\begin{multline*}
\sigma(x,y)=0.3(1-x)^2 \exp(-x^2-(y+1)^2)\\
-(\frac{x}{5}-x^3-y^5)\exp(-(x^2+y^2))-\frac{1}{30}\exp(-(x+1)^2-y^2).
\end{multline*}
Define $q(k_0,x)={\rm i} \sigma(3x,3y-1)/k_0$ in $\Omega$. See figure \ref{Example1} for the surface
plot of the scatterer function in the
domain $|x|<1$, the result of the final reconstructions with noise-free and noisy data using the wavenumber $k_0=10.1$, which has relative error $1.64\%$ and $2.49 \%$, respectively.
\end{example}

\begin{example}

In this example, we performed the reconstruction with scatterer function
\begin{equation*}
q(k_0,x)=
\begin{cases}
{\rm i} \displaystyle \frac{0.2}{k_0} & (x+0.25)^2+y^2<0.2^2 \quad \text{in $\Omega$,}\\
{\rm i} \displaystyle \frac{0.2}{k_0} & (x-0.25)^2+y^2<0.2^2 \quad \text{in $\Omega$,}\\
0 & \text{elsewhere in $\Omega$.}
\end{cases}
\end{equation*}
Using the wavenumber $k_0=12.1$, the final reconstructions with noise-free and noisy data have relative
errors $5.23\%$ and $16.05 \%$. The functions and reconstructions are shown in figure \ref{Example2}.

\end{example}

\section{Conclusions}
In this paper we have presented a continuation method on an inverse scattering 
problem in dispersive media using multi-frequency data The numerical results showed the 
efficiency and robustness of the algorithm. Our future project is to vary the real part of the dielectric constant, in which case both real and imaginary part of the scatterer need to be reconstructed. Another direction is to consider the three dimensional case.

\section*{Acknowledgement}
The author would like to sincerely thank Gang Bao for his generous guidance and help on the topic of inverse scattering problem and Peijun Li for his instructive idea of tomography problem. 

\section*{Funding}
The author was supported by the Faculty Development Grant from Saint Francis University.

\end{document}